\documentclass[11pt, onesdie,reqno]{amsart} 
 \usepackage{amsaddr}  
\usepackage{cite}  
\usepackage[noend]{algpseudocode} 
\usepackage{mathrsfs}    
\usepackage{tikz-cd}

\makeatletter
\newcommand*\rel@kern[1]{\kern#1\dimexpr\macc@kerna}
\newcommand*\WB[1]{%
  \begingroup
  \def\mathaccent##1##2{%
    \rel@kern{0.8}%
    \overline{\rel@kern{-0.8}\macc@nucleus\rel@kern{0.2}}%
    \rel@kern{-0.2}%
  }%
  \macc@depth\@ne
  \let\math@bgroup\@empty \let\math@egroup\macc@set@skewchar
  \mathsurround\z@ \frozen@everymath{\mathgroup\macc@group\relax}%
  \macc@set@skewchar\relax
  \let\mathaccentV\macc@nested@a
  \macc@nested@a\relax111{#1}%
  \endgroup
}
\makeatother

\newtheorem{corollary}{Corollary} 
\newtheorem{definition}{Definition} 
\newtheorem{lemma}{Lemma} 
\newtheorem{proposition}{Proposition}  
  
\newtheorem{theorem}{Theorem}

\newcommand{\WH}[1]{\widehat{#1}  }
\newcommand{\WT}[1]{\widetilde{#1}  }
  
\newcommand{\RR}{\mathbb{R}}  
\newcommand{\RRRd}{\RR^{3d}}  
\newcommand{\SL}{SL(3,\mathbb{Z})} 
\newcommand{\LRd}{ \WT{\mathcal{L}}_{{t}  }^d\times \RRRd}   
\newcommand{\LRdt}{   \WH{\LL}_{t}^d\times\RRRd}     
\newcommand{\LL}{\mathcal{L}}  
\newcommand{\KK}{\mathcal{K}}  
\newcommand{\toro}{T\mathbb{T}}   
\newcommand{\MM}{\mathcal{L}_0^d\times \RRRd}

\newcommand{\norm}[1]{\left\lVert#1\right\rVert}
 
\newcommand{\normll}[1]{\left\lVert#1\right\rVert_2}
\newcommand{\norml}[1]{\left|#1\right|}  
\newcommand{\normLK}[1]{ \left\lVert#1\right\rVert_{ L_{{\KK}_n}^{\infty}}}

\newcommand{\pp}{\mathbf{p}}  
\newcommand{\ppb}{\WB{\pp}}  
\newcommand{\ppt}{\WT{\pp}}     
\newcommand{\qq}{\mathbf{q}}  
\newcommand{\qqb}{\WB{\qq}} 
\newcommand{\qqt}{\WT{\qq}}  
\newcommand{\qqh}{\widehat{\qq}}  
\newcommand{\pph}{\widehat{\pp}}

\newcommand{\ff}{\mathbf{f}}   
\newcommand{\xx}{\mathbf{x}}
  
\newcommand{\xxb}{\WB{\xx}}
\newcommand{\xxh}{\widehat{\xx}}  

\newcommand{\XX}{\mathbf{X}}
\newcommand{\bb}{\mathbf{b}}
\newcommand{\XXt}{\widetilde{\XX}}
\newcommand{\XXh}{\widehat{\XX}}
\newcommand{\bbt}{\WT{\bb}} 
\newcommand{\XXb}{\WB{\XX}}  
\newcommand{\yy}{\mathbf{y}}
  
\newcommand{\yyh}{\widehat{\yy}}    
\newcommand{\yyb}{\WB{\yy}}

\newcommand{\EEys}[1]{\mathbb{E}^{{s}  ,\yy}{#1}  } 
\newcommand{\EEF}[1]{\mathbb{E} \Big[\ #1\Big| \mathcal{F}_k  \Big]\ }
\newcommand{\AAI}{ {\Gamma} }

\title[NELD Convergence]{Convergence of Nonequilibrium Langevin Dynamics
for Planar Flows} 
\author{Matthew Dobson and Abdel Kader Geraldo}  
\address{Department of Mathematics, University of Massachusetts Amherst, 
Amherst, MA 01003}
\email{dobson@math.umass.edu}
\date{\today}

\begin{document} 
 
\maketitle     
\tableofcontents
 
\begin{abstract}  
We prove that incompressible two dimensional   nonequilibrium Langevin dynamics (NELD)  converges exponentially fast to a steady-state limit cycle. We use automorphism remapping  periodic boundary conditions (PBCs) techniques such as Lees-Edwards PBCs  and  Kraynik-Reinelt   PBCs  to treat respectively shear flow  and planar elongational flow. After rewriting NELD in   Lagrangian coordinates, the convergence is shown using a  technique similar to [ R. Joubaud, G. A. Pavliotis, and G. Stoltz,2014]. 
\end{abstract}

 \section{Introduction} 
A wide range of nonequilibrium molecular  dynamics (NEMD) techniques~\cite{Evans, todd17} are used in the study of molecular fluids under steady flow, and some recent applications can be found in \cite{Lang-phil,Oconnor-Ge,Oconnor-Al,Nicholson,Oliveira,Nishioka-Akihiro,Templeton,Menzel-Daivis-Todd,Ewen,Li-Zhen,DAIVIS,XuWen,Baranyai,todd1998,todd2000}. 
Here we study the exponential convergence of the probability density of Nonequilibrium Langevin dynamics (NELD)   under   incompressible two dimensional flows such as shear flow and planar elongational flow with spatial periodic boundary conditions (PBCs).  We consider  a molecular system corresponding to the 
micro-scale motion of a fluid with local strain rate $\nabla \textbf{u}$ and  denote the steady background flow matrix of the molecular system by
$A=\nabla \textbf{u} \in \RR^{3 \times 3}$. The coordinates of the simulation box are given by three linearly independent column vectors
coming from the origin, and we write them in a matrix
\begin{align*}
    L_{t}=\begin{bmatrix}\mathbf{v}_{t}^1 &\mathbf{v}_{t}^2 &\mathbf{v}_{t}^3\end{bmatrix}\in \RR^{3 \times 3}, \qquad t \geq 0.
\end{align*} 
The initial simulation box is given by $L_0$ and  the  lattice deforms with the background flow according to the equation 
\begin{align*}
  L_{t}  = {e^{t A}} L_0 .
\end{align*}

If $L_0$ is not chosen appropriately, the simulation box can become extremely stretched with degenerate geometry.  For example, in the elongational flow case, if the compression is parallel to one of the edges of the simulation box, then the box will become degenerate to a point where a particle and its image become arbitrarily close. Thus in order to perform a long simulation, we consider specialized PBCs which consist of using a lattice automorphism represented as a $3 \times 3$ integer matrix with determinant one to remap the simulation box at various times during the simulation.    These types of PBCs were
first used in the shear flow case by Lees and Edwards (LE)~\cite{LE} and were then later extended to the planar elongational flow case 
 by   Kraynik and Reinelt (KR)~\cite{KR}. The analog of these types of PBCs which treat three dimensional flows such as  uniaxial flow, biaxial flow, and generalized three-dimensional diagonalizable flow can be found in \cite{R-KR,Dobson,Hunt}.

The NELD equation is derived in \cite{McPhie, dlls}. We express it in terms of the relative momentum $\ppt$ as 
\begin{align}\label{neld2x}
\begin{cases}
d{\qqt}&=  ({\ppt}+A\qqt) {dt},\\
d{\ppt}&= -\nabla V({\qqt}) {dt} -\gamma  {\ppt} {dt}+\sigma dW,
\end{cases}
\end{align}
where $\sigma^2=\frac{2\gamma}{\beta}$ is the fluctuation coefficient, with $\beta$ as the inverse temperature, and $V$ is the potential. We assume that the gradient of the potential is finite. The position and the momentum of the particles are denoted respectively by $(\qqt,\ppt)\in  \WT{\mathcal{L}}_t^{d}\times\RRRd$,  where the set
\begin{align}\label{domain-euler-absolute}
 \WT{\mathcal{L}}_t& = \{ L_t \mathbf{x} \  {\big|} \   \mathbf{x}  \in \mathbb{T}^3    \}
\end{align}
 defines the time dependent simulation box.
  Note that when there is no background flow, \eqref{neld2x} becomes equilibrium Langevin Dynamics 
\begin{align*} 
\begin{cases}
d{\qq}&=  {\pp} {dt},\\
d{\pp}&=-\nabla V({\qq}) {dt} -\gamma  {\pp} {dt}+\sigma dW, 
\end{cases} 
\end{align*}
with $(\qq,\pp)\in  {\mathcal{L}}_0^{d}\times\RR^{3d}$.   
  
 It  has been shown (see for instance \cite{Talay,Mattinglya2002,cances,neld-grabriel,lelievre-stoltz,luc-ergo})  that under suitable conditions, the equilibrium Langevin Dynamics is ergodic with respect to the Boltzmann-Gibbs distribution 
\begin{align*}
\nu(\qq,\pp)d\qq d\pp=\frac{1}{Z}e^{-\beta H(\qq,\pp)} d\qq d\pp, \quad Z=\int_{\LL_0^{d}\times\RR^{3d}}e^{-\beta H(\qq,\pp)} d\qq d\pp,
\end{align*}
where   $Z$ is the normalization constant, and $H$ is the Hamiltonian of the system given by
\begin{align*}
H(\qq,\pp)=\frac{1}{2}\left\langle\pp,\pp\right\rangle+V(\qq).
\end{align*}
However, convergence to a limiting measure has not been established for NELD under moving domains.   In this paper, we show the  existence, uniqueness, and exponential convergence  of the NELD to a limit cycle following the work done in~\cite{Joubaud}.  
In Section~\ref{sect}, we formulate the NELD in Eulerian and Lagrangian coordinates and the main result of the paper is in Section~\ref{sectt}, where we prove the convergence of the NELD to a probability density function in Proposition~\ref{prop:converge}. 

 \section{Reformulation of NELD in Lagrangian  coordinates}\label{sect}
In this section, we rewrite the NELD equation~\eqref{neld2x}  in Lagrangian coordinates.  We then define the remapped Eulerian domain  under shear flow and planar elongational flow using the LE and KR PBCs respectively in Section~\ref{domain-def}.  In order to distinguish from the remapped coordinates, we refer to the original coordinate systems as ``absolute Eulerian coordinates'' and ``absolute Lagrangian coordinates.''   We then derive the NELD in the remapped coordinates in Section~\ref{neld-lagrangian-derivation}.
 We start by considering   the change of variables from the absolute Eulerian to Lagrangian coordinates
\begin{align}\label{change1}
\begin{cases}
\qq  &= {e^{-t A}} \qqt , \quad \qq\in\LL_0^{d}, \ \qqt\in\WT{\mathcal{L}}_t^d \\
\pp  &={e^{-t A}} \ppt , \quad \pp\in\RRRd.
\end{cases} 
\end{align}
Computing the time derivative of the position in \eqref{change1}, and using \eqref{neld2x}, we have
\begin{align*}
 {d } \qq  &=  {e^{-t A}}\big( {d \qqt}-A\qqt{{dt}}\big)  =\pp {dt},  
\end{align*}
 taking the time derivative of $\pp$ gives
\begin{align*}
 {d } \pp  &=  {e^{-t A}}\big( {d \ppt}-A\ppt{{dt}}\big) =  -{e^{-t A}}{{\nabla {V}({e^{t A}} {\qq})}} {dt} -\AAI \pp {dt}+   \sigma {e^{-t A}} dW ,
\end{align*}
where $\AAI=(\gamma+A)$. Thus, the NELD equation in  Lagrangian coordinates is written as
\begin{align*}
\begin{cases}
d \qq&= \pp {dt},\\
d\pp&= -{e^{-t A}}{{\nabla {V}({e^{t A}} {\qq})}} {dt} -\AAI \pp {dt}+   \sigma {e^{-t A}} dW,
\end{cases}
\end{align*}
Before we derive the NELD equation in remapped Lagrangian coordinates under LE and KR PBCs, we give some background on those PBCs.

\subsection{Remapping the Unit Cell}\label{domain-def}
We start by defining  the remapped Eulerian domain under the shear flow followed by the planar elongational flow case. 
\subsubsection{Shear Case}
We denote the background  matrix of  the shear flow    by
\begin{align*}
A=\begin{bmatrix}
0 & \epsilon & 0 \\
0 & 0 & 0 \\
0 & 0 & 0
\end{bmatrix}, \  \epsilon\in\RR^{\ast}.
\end{align*}
At a time $t$, the basis vectors for the simulation box
are the columns of the matrix  
\begin{align*} 
L_t= \begin{bmatrix}1& t \epsilon& 0\\0&1&0\\0&0&1\end{bmatrix}  L_0  \textrm{ where }L_0 = \begin{bmatrix}1&  0& 0\\0&1&0\\0&0&1\end{bmatrix}.
\end{align*}
Since $L_t$ is highly sheared as $t$ becomes large, the  interparticle interaction computation becomes more difficult. We can prevent  this anomaly by applying the LE PBCs which consists in  
multiplying $L_t$ by the lattice automorphism matrix  
\begin{align*}
M^k=\begin{bmatrix}1& -1& 0\\0&1&0\\0&0&1\end{bmatrix}^k=\begin{bmatrix}1& -k& 0\\0&1&0\\0&0&1\end{bmatrix}, \quad k\in \mathbb{Z},
\end{align*}
to get the remapped simulation box lattice 
\begin{align*}
L_t M^k=\begin{bmatrix}1& t\epsilon-k& 0\\0&1&0\\0&0&1\end{bmatrix}, k \in \mathbb{Z}.
\end{align*}
Since $M$ is an integer matrix with determinant equal to one 
($M \in \SL$), the lattice basis vectors in $L_t$ and $L_t M^k$ generate the same lattice.
By choosing  $k=-\left \lfloor t \epsilon\right \rceil $, where $\left \lfloor x\right \rceil $ denotes $x$  rounded to nearest integer, we ensure
that the stretch is at most half of the simulation box. Then we observe that the stretch   matrix is time-periodic with the period $T=\frac{1}{\epsilon}$ :
\begin{align}\label{stretch}
\begin{bmatrix}0& t\epsilon-\left \lfloor t\epsilon\right \rceil & 0\\0&0&0\\0&0&0\end{bmatrix} = [t]\begin{bmatrix}0& \epsilon& 0\\0&0&0\\0&0&0\end{bmatrix} = [t]A, \quad\textrm{where} \ [t]\equiv t \ \mathrm{mod}\ T.
\end{align}
This implies that the particle position belongs to  remapped Eulerian domain  
\begin{align}\label{domain1}
 \WH{\LL}_{t}&=\{   {e^{[t]A}}   L_0\mathbf{x}  {\big|} \mathbf{x}  \in \mathbb{T}^3 \},\ \textrm{where} \ \mathbb{T}^3=\RR^3 \backslash \mathbb{Z}^3.
\end{align}
In the Section~\ref{remm}, we analyse the particles remapped position in $ \WH{\LL}_{t}$ and in the  unit cell. 
\subsubsection{Planar Elongational Flow case}
We consider the planar elongational flow (PEF) case with 
background flow matrix given by 
\begin{align*}
A=\begin{bmatrix}
\epsilon &0 & 0 \\
0 & -\epsilon & 0 \\
0 & 0 & 0
\end{bmatrix}, \  \epsilon\in\RR^{\ast},
\end{align*}
which means that the simulation box elongates in the $x$ direction and shrinks in the $y$ direction of the standard coordinate plane. As $t$ goes to infinity,   a particle and its image can become arbitrarily close if an edge of the simulation box is aligned with the $y$ coordinate.  This would lead to a breakdown in the simulation. We prevent this issue by applying the KR PBCs, which consists in carefully choosing the alignment of the initial simulation box and remapping the simulation box with   a  matrix  ${M}\in \SL$.  We choose $M$ such that it is diagonalizable with eigenvalues of the form 
\begin{align*}
M  S=S \Lambda , \quad \Lambda= \begin{bmatrix}
\lambda &0 & 0 \\
0 & \lambda^{-1}  & 0 \\
0 & 0 &1
\end{bmatrix}  
, \quad \lambda > 0,  \quad \lambda \neq 1.
\end{align*} 
We choose the initial lattice $L_0=S^{-1}$  rather than 
the standard coordinate directions, and this will prevent particle
images from becoming too close. 
If we remap the lattice basis vectors by applying $M^k,$ 
we get
\begin{align*}
L_{t } M^k= e^{ t A} L_0  M^{k} =  e^{t \epsilon D} \Lambda^n  S^{-1}= e^{ (t\epsilon + k\eta ) D }S^{-1} \textrm{, where }D= \begin{bmatrix}
1 &0 & 0 \\
0 &-1 & 0 \\
0 & 0 &0
\end{bmatrix},\quad\eta=\log(\lambda).
\end{align*}
Letting $T=\frac{\eta}{\epsilon}$, 
the  stretched matrix $[t]A$ and the position domain of the particles are also respectively  expressed  in the periodic form as in ~\eqref{stretch} and   \eqref{domain1}.

 \subsection{Remapped coordinates}  \label{remm} 
Note that under the PBCs, the remapping of the simulation box is followed by remapping the particle positions to lie within the simulation box.  In order to include this remapping in the dynamics, we write down the remapping function for both particle positions and momenta in both Eulerian and Lagrangian coordinate systems.  

\subsubsection{Remapping particle positions in Eulerian coordinates}
Here, we define the function which remaps the particle positions from the absolute Eulerian domain $\WT{L}_t$~\eqref{domain-euler-absolute} to the remapped Eulerian domain $\WH{L}_t$~\eqref{domain1}.  This function chooses the image particle that lives within the unit
cell for the remapped lattice.   
We start by defining the modulus operation applied to each vector component 
\begin{align*}
{ {\mathbf{g}}(\xx)  \equiv \xx \ \mathrm{mod}\  1 }, \textrm{ where } \ \xx \in \RR^3.
\end{align*}
We then compute $ e^{-[t] A}  \qqt,$ which maps the particles back to a point in time when an integer number of periods have occurred.  Multiplying by $L_0^{-1}$ expresses the particle position in lattice coordinates, where the coordinates corresponding to unit cell is given by $[0,1]^3,$ but the particle coordinates may be outside this cube due to the cell's deformation.  Applying the modulus operation $\mathbf{g}$ finds the coordinates of images within the unit cell, and then we map back to Eulerian space by multiplying by   $e^{   [t]A } L_0,$ which gives 
\begin{align*} 
\WH{\mathbf{g}}_t  &: \WT{\LL}_t \rightarrow \WH{\LL}_t 
\qquad
\qqh = \WH{\mathbf{g}}_t(\qqt)= e^{   [t]A } L_0  \,  {\mathbf{g}}  \left(  L_0^{-1}  e^{-[t] A}  \qqt \right).
\end{align*}

  \subsubsection{Remapping particle positions in Lagrangian coordinates}\label{sect_remap}
We use the Eulerian remapping function to remap the Lagrangian coordinates, as shown in the following diagram:
 \begin{equation*}
    \begin{tikzcd}
\WT{\LL}_t\arrow{r}{\WH{\mathbf{g}}_t} & \WH{\LL}_t\\
\LL_0\arrow{u}{e^{tA}} \arrow{r}{\WB{\mathbf{g}}_t} & \LL_0
 \arrow{u}[swap]{e^{[t]A}} 
\end{tikzcd}   
\end{equation*}
We map from absolute Lagrangian coordinates to Eulerian coordinates using~\eqref{change1} , perform the remapping $\WH{\mathbf{g}}_t$, and then returning to Lagrangian coordinates:
\begin{align*} 
\WB{\mathbf{g}}_t  &: {\LL}_0 \rightarrow {\LL}_0 
\qquad \qqb = \WB{\mathbf{g}}_t(\qq)
= e^{-[t]A} \, \WH{\mathbf{g}}_t \left(  e^{t A}  \qq \right).
\end{align*}
 
\subsubsection{Remapping particle momenta}
We now remap the momentum of the particles from the absolute to the remapped domains. 
 We start by taking the time derivative of $\qqb$, we get
\begin{align*}
 d\qqb=d\WB{\mathbf{g}}_t(\qq) = e^{ \left\lfloor {\frac{t}{T}} \right\rfloor TA}\pp dt =\ppb dt,
\end{align*}
where we have defined the remapped momentum 
 \begin{align}\label{momentumRemapEuler}
 \WB{\mathbf{h}}_t  &:  \RRRd  \rightarrow \RRRd
\qquad \ppb = \WB{\mathbf{h}}_t(\pp)= e^{ \left\lfloor {\frac{t}{T}} \right\rfloor TA}\pp.
\end{align}
We chose the above mapping based on the form of the position equation $ d\qqb= \ppb dt,$ and we now show that this choice leads to consistency of the NELD equations in remapped Eulerian coordinates.  In remapped Eulerian coordinates, the domain is time periodic but the coefficients of the NELD are not changed by the remapping.
We denote the position in the remapped Eulerian domain by 
\begin{align*}
\qqh =e^{[t]A}\qqb, \quad \qqh \in \WH{\LL}_t,  
\end{align*}
and then by taking the time derivative of $\qqh$ over a  periodic time $\theta =[t] \in [0,T)$ , we get
\begin{align*} 
 {d } \qqh  &=  {e^{\theta A}}\big( {d \qqb}+A\qqb d\theta\big)  ={e^{\theta A}}(\ppb+A\qqb) d\theta=(\pph+A \qqh)d\theta, \label{ew1mn} 
\end{align*} 
where we have defined the relative momentum of the particles in the remapped Eulerian coordinates by
\begin{align*}
\pph = e^{ [t]A}\ppb.
\end{align*}
Since $\ppt =e^{tA} \pp$ by \eqref{change1}, and by using \eqref{momentumRemapEuler}, the relative momentum of the particles in the Eulerian domain are the same, as 
\begin{align*} 
\pph = e^{ [t]A}\ppb =e^{tA} e^{- \left\lfloor {\frac{t}{T}} \right\rfloor TA}\ppb = e^{tA} \pp =\ppt.
\end{align*}
Thus the NELD as a function of $(\qqh_{kT+\theta},\pph_{kT+\theta})\in \WH{\LL}_t^d\times\RRRd$ which is written as
 \begin{align*} 
\begin{cases}
d   \qqh_{kT+\theta}&=  \big( \pph_{kT+\theta} +A\qqh_{kT+\theta}\big){d\theta} \\
d\pph_{kT+\theta}&=- \nabla {V}\big(\qqh_{kT+\theta}\big)  {d\theta} -  \gamma \pph_{kT+\theta} {d\theta}+   \sigma  dW_{kT+\theta} ,
\end{cases}
\end{align*}
 has the same coefficients as \eqref{neld2x}.
 
Now let us define the momentum remapping function in the Eulerian domain  by
 \begin{align*}
   \WH{\mathbf{h}}_t: & \ \RRRd  \to \RRRd \quad \WH{\mathbf{h}}_t(\ppt)=\ppt .
\end{align*}
We summarize the results in this section using the following definitions:
\begin{definition}
   The remapping function from the absolute   to the remapped Lagrangian coordinates is given by 
   \begin{align*}
{\mathfrak{R}}_t:& \ \MM\to\MM \qquad {\mathfrak{R}}_t((\qq_t,\pp_t))=\big( \WB{\mathbf{g}}_t(  \qq_{t}),
\WB{\mathbf{h}}_t(\pp_t) \big) .
\end{align*}
\end{definition} 
\begin{definition}
   The remapping function from the absolute to the remapped Eulerian coordinates is given by
\begin{align} 
\WT{\mathfrak{R}}_t:&  \ \WT{\LL}_t^d\times\RRRd \rightarrow \WH{\LL}_t^d\times\RRRd \qquad \WT{\mathfrak{R}}_t((\qqt_t,\ppt_t))=\big({\WH{\mathbf{g}}_t}(  \qqt_{t}) ,
 \WH{\mathbf{h}}_t(\ppt_t)\big).
\end{align}
\end{definition} 
Using the above definitions, and the linear map
\begin{align}
 \Phi_{t} &: \MM \to \LRdt \qquad
  \Phi_t(\XXb_{t})=\begin{bmatrix}e^{ [t]A}  &0\\ 0  &e^{[t]A} \end{bmatrix}\XXb_t,\label{change-var-}\\ 
 \WT{\Phi}_t &: \MM \to \LRd \qquad
  \WT{\Phi}_t(\XX_t)=\begin{bmatrix}e^{t A}  &0\\ 0  &e^{t A} \end{bmatrix} \XX_t,  \label{change-var}
\end{align}
we have the commutative diagram as following: 
\begin{equation*}
    \begin{tikzcd}
\LRd\arrow{r}{\WT{\mathfrak{R}}_t} & \LRdt\\
\MM\arrow{u}{\WT{\Phi}_t} \arrow{r}{{\mathfrak{R}}_t} &  \MM
 \arrow{u}[swap]{{\Phi}_t} \\ 
\end{tikzcd}   
\end{equation*} 
The remapping function from the remapped Lagrangian  to the remapped Eulerian domain is given by
\begin{align}\label{Linear-map}
 \Phi_{t}=\WT{\mathfrak{R}}_t\circ\WT{\Phi}_t\circ{\mathfrak{R}}_t^{-1}.
\end{align}

\subsection{NELD with LE and KR  PBCs in the Lagrangian coordinates}\label{neld-lagrangian-derivation}

Using the function ${\mathfrak{R}}_t$, we derive the NELD in the remapped Lagrangian domain. The coefficients of the NELD are time periodic, and we use this periodicity in the analysis that follows. In fact, we use ${\mathfrak{R}}_t$ to consider  the change of variables  
\begin{align*}
\begin{cases}
    \qqb_{kT+\theta} = \WB{\mathbf{g}}_k(  \qq_{kT+\theta}), \\
 \ppb_{kT+\theta}=  e^{ k TA}  \pp_{kT+\theta},
\end{cases}
\end{align*}
where   $\theta=[t] \in [0,T)$ and    $t\in [kT,(k+1)T)$.
We use $\WB{\mathbf{g}}_k$ to define $\WB{\mathbf{g}}_t$ for   $t\in [kT,(k+1)T)$. 
Changing variables in the position equation of the NELD  gives
\begin{align*}
\frac{d}{{d\theta}}\qqb_{kT+\theta}=&\frac{d}{{d\theta}} \WB{\mathbf{g}}_k (  \qq_{kT+\theta} ) ={\nabla }_{\qq} \WB{\mathbf{g}}_k (\qq_{kT+\theta} )  \frac{d}{{d\theta}} \qq_{kT+\theta}=e^{ k TA}  \pp_{kT+\theta} =  \ppb_{kT+\theta}  ,
\end{align*}
while in the momentum, we have
\begin{align*}
d\ppb_{kT+\theta}&= e^{ k TA}  d\pp_{kT+\theta} 
=  - e^{ - \theta A } (\nabla {V}\circ\WT{\mathbf{g}}_k)(  e^{ \theta A } \qqb_{kT+\theta})  {d\theta} -\Gamma\ppb_{kT+\theta} {d\theta}+   \sigma  e^{ - \theta  A }  dW_{kT+\theta}.
\end{align*}
Thus the NELD  as a function of  $(\qqb_{kT+\theta},\ppb_{kT+\theta})\in\LL_0^d\times\RRRd$ is given by  
\begin{align} \label{NELD-22-3}
\begin{cases}
d   \qqb_{kT+\theta} &=  \ppb_{kT+\theta} {d\theta} ,\\
d\ppb_{kT+\theta}&= - e^{ - \theta A } \nabla {V}\big(e^{ \theta A }   \qqb_{kT+\theta}\big) {d\theta} -\Gamma\ppb_{kT+\theta} {d\theta}+   \sigma  e^{ - \theta  A }  dW_{kT+\theta}.
\end{cases}
\end{align}
       
\section{Ergodicity of NELD under Planar Flow}\label{sectt}
We  start this section by deriving the forward and backward Kolmogorov equation of the NELD in Section~\ref{sectProperty}, then in Section~\ref{sectProposition} we prove the main result of the paper, the convergence of the NELD to a limit cycle.    
 \subsection{Fokker-Planck Equation of NELD in Eulerian and Lagrangian Coordinates}\label{sectProperty} 
 \label{generator}

We now derive  the forward and backward Kolmogorov equations  of   NELD in Eulerian and Lagrangian  coordinates.  In the absolute Eulerian coordinates, the NELD equations of motion are rewritten as
\begin{align}\label{SDE-eurler}
\begin{cases}
d \XXt_{{t}  }=\bbt(\XXt_{{t}  }) {dt} +\WT{\Sigma} dW_{{t}  },\quad \XXt_{{t}  }\in \LRd ,\\
\XXt_{{t}  }=\begin{bmatrix}\qqt\\\ppt\end{bmatrix},\quad \bbt(\XXt_{{t}  })=\begin{bmatrix}\ppt+A\qqt \\ ( -\nabla V({\qqt})  -\gamma  {\ppt}  \end{bmatrix},\quad \WT{\Sigma} =\begin{bmatrix}0 & 0\\0 & \sigma\end{bmatrix}.
\end{cases} 
\end{align} 
Given any smooth function $\WT{\ff} \in C^{\infty}(\LRd ; \RR)$  It\^o's lemma says that 
\begin{align}\label{EW}
d\WT{\ff}(\XXt_{{t}  }) = \big( \WT{U}_t \WT{\ff}\big)(\XXt_{{t}  })dt+\left\langle\nabla\WT{\ff}(\XXt_{{t}  }),\WT{\Sigma} {dW} \right\rangle, 
\end{align}  
where
\begin{align*}
 {\WT{U}_t}=\left\langle\ppt+A\qqt , \nabla_{\qqt} \cdot\right\rangle+\left\langle-\nabla V(\qqt)   ,\nabla_{\ppt}\cdot\right\rangle-\gamma\left\langle\ppt, \nabla_{\ppt}\cdot\right\rangle +\frac{1}{2} \sigma  \sigma^T:\nabla^2 .
\end{align*}
The symbol $:$ and $\left\langle\cdot,\cdot\right\rangle$ denote the Frobenius inner product.
Moreover, we derive the generator ${U}_t$ of the NELD in the absolute Lagrangian coordinates as follows: rewriting~\eqref{NELD-22-3} in the absolute Lagrangian coordinates as
\begin{align}\label{SDE}
\begin{cases}
d \XX_{{t}  }=\mathbf{b}(\XX_{{t}  },{t}  ) {dt} +\Sigma({t}  )dW_{{t}  },\quad \XX_{{t}  }\in \MM \\  
\XX_{{t}  }=\begin{bmatrix}\qq\\\pp\end{bmatrix},\quad \bb(\XX_{{t}  },{t}  )=\begin{bmatrix}\pp\\ -{e^{-t A}} {{\nabla {V}({{e^{t A}}  } {\qq})} } -\AAI \pp \end{bmatrix},\quad \Sigma({t}  )=\begin{bmatrix}0 & 0\\0 & \sigma{e^{-t A}} \end{bmatrix}.
\end{cases} 
\end{align}
When applying It\^o's lemma, we compute \eqref{EW} using the
change of variables  as 
\begin{align*}
  d(\WT{\ff}\circ\WT{\Phi}_t)( \XX_t) 
 = &  \big(\partial_t+{U}_t\big)(\WT{\ff}\circ\WT{\Phi}_t)( \XX_{{t}  })dt+\left\langle\nabla_{\XX}(\WT{\ff}\circ\WT{\Phi}_t)( \XX_{{t}  }),\Sigma({t}  )  {dW} \right\rangle, 
\end{align*}  
where  
\begin{align*}
 {{U}_t}=\left\langle\pp,\nabla_{\qq} \cdot\right\rangle-\left\langle {e^{-t A}} {\nabla {V}({{e^{t A}}  } {\qq})} ,\nabla_{\pp} \cdot \right\rangle-\left\langle\AAI\pp,\nabla_{\pp}\cdot\right\rangle+\frac{1}{2}(\sigma{e^{-t A}} ) (\sigma{e^{-t A}} )^T:\nabla^2.  
\end{align*}       
Now let us write  the strong solution  of \eqref{SDE-eurler} as
\begin{align*}
\XXt_{{t}  }-\XXt_{{s}  }&=\int_{{s}  }^{{t}  }\bbt(\XXt_{{u}  }) {du} +\WT{\Sigma}  dW_{{u}  }
,
\end{align*}
 and   define the density transition function from one state to another in the continuous-time Markov chain $\XXh_t=\WT{\mathfrak{R}}_t(\XXt_{t}) $ by 
\begin{align*}
\WH{\mathcal{P}}_t(\yyh,\WH{B}_t) =\mathbb{P}\big(\XXh_{{t}  }\in \WH{B}_t {\big|} \XXh_{{s}  }=\yyh\big)=\int_{\WH{B}_t} \WH{\psi}(t,\xxh{\big|} s,\yyh)d\xxh, \textrm{ where }  \WH{\psi}({t}  , \xxh \, | {t}  , \yyh) = \delta(\xxh-\yyh),
\end{align*}
 $ \forall \xxh\in \LRdt, \ \WH{B}_t\in \mathcal{B}(\LRdt)$. Here  $  \mathcal{B}(\LRdt)$ denotes the B\"orel $\sigma$-algebra on  $\LRdt$. 
We transform to the remapped Eulerian coordinates using $\WT{\mathfrak{R}}_t$. 
Using \eqref{Linear-map}, we define
\begin{align*}
 {\ff} : \LRdt \to \RR \quad \ff=\WT{\ff} \circ  \WT{\mathfrak{R}}_t^{-1},
\end{align*}
such that
\begin{align*}
\ff\circ\Phi_{t}\circ{\mathfrak{R}}_t=\WT{\ff} \circ  \WT{\mathfrak{R}}_t^{-1}\circ\Phi_{t}\circ{\mathfrak{R}}_t= \WT{\ff} \circ \WT{\Phi}_t  : \MM \to \RR .
\end{align*}
Then  in the remapped Eulerian domain $\WH{\LL}_t$, we define the expectation of $  {\ff}  : \LRdt \to \RR$ with respect to the probability density function above by
\begin{align} \label{expectationE} 
{\phi}({s}  ,\yyh)=\EEys{\ff(\XXh_{{t}  })}=&\int_{\LRdt}\ff(\xxh) {\psi}(t,\xxh {\big|} {s}   ,\yyh) d\xxh. 
\end{align}   
Similarly, in the absolute Lagrangian domain, let us write  the strong solution  of \eqref{SDE} as
\begin{align*}
\XX_{{t}  }-\XX_{{s}  }&=\int_{{s}  }^{{t}  }\bb(\XX_{u}) {du} + {\Sigma}(u)  dW_{u},
\end{align*}
 and   define the density transition function from one state to another in the continuous-time Markov chain $\XXb_t={\mathfrak{R}}_t(\XX_{{t}  }) $ by
\begin{align*}
\WB{\mathcal{P}}_t(\yyb,\WB{B}) =\mathbb{P}\big(\XXb_{{t}  }\in \WB{B}  {\big|} \XXb_{{s}  }=\yyb\big)=\int_{\WB{B}} {\psi}(t,\xxb{\big|} s,\yyb)d\xxb, \textrm{ where }  {\psi}({t}  , \xxb \, | {t}  , \yyb) = \delta(\xxb-\yyb),
\end{align*}
 $ \forall \xxb\in \MM, \ \WB{B} \in \mathcal{B}(\MM)$ .
We define a measurable function $  {\ff}\circ\Phi_{t} : \MM \to \RR$. 
 Using the change of variables~\eqref{change-var-}, we can rewrite the  expectation \eqref{expectationE}    with respect to the probability density function above as
\begin{align*}  
 \big({\phi}\circ\Phi_{s}\big)(s, \yyb)=\EEys{(\ff\circ\Phi_{t})(\XXb_{{t}  })}=&\int_{\MM}(\ff\circ\Phi_{t})(\xxb) {\psi}(t,\xxb {\big|} {s}   ,\yyb)d \xxb  ,\quad s\leq t 
.
\end{align*}
Now we  derive  the backward Kolmogorov equation for the NELD in following lemma:
\begin{lemma}\cite[Theorem 6.1]{friedman}
The backward Kolmogorov equation for the NELD is
\begin{align*} 
\begin{array}{ccc}
   { \partial_{{s}  }{\psi}(t,\xxb{\big|} {s}  ,\yyb) +\big({U}_{s} {\psi}\big)(t,\xxb{\big|} {s}  ,\yyb)}=0,     &    \textrm{where}     & \left. {\psi}(t, \xxb \, | {s}  , \yyb)\right|_{t=s} = \delta(\xxb-\yyb) ,  \\
   { \partial_{{s}  }\WH{\psi}(t,\xxh\big| s,\yyh) +\big(\widehat{U}_{s} \widehat{\psi}\big)(t,\xxh\big| s,\yyh)}=0,   &     \textrm{where}     &   \left.\WH{\psi}(t, \xxh ,\big| s,\yyh)\right|_{t=s} = \delta(\xxh-\yyh) , \ {s}   <t .\nonumber
\end{array} 
\end{align*}
\end{lemma} 

The forward Kolmogorov equation of the NELD is given  in the following lemma: 
\begin{lemma}\label{forward-kolmo-lemma}
For  an operator $G$,   let us  denote the adjoint operator with respect to Lebesgue measure by $G^{\dagger}$.
The forward Kolmogorov equation of the NELD is
\begin{align}\label{forward-kolmo}
    \big(-\partial_{t}\psi +{U}_{t}^{\dagger}\psi\big) (t,\xxb{\big|} {s}  ,\yyb)=0 \textrm{ and } \big(-\partial_{t}\WH{\psi} +\WH{U}^{\dagger}_t\WH{\psi}\big)  (t,\xxh \big| s,\yyh) =0 .
\end{align}
\end{lemma}
\begin{proof}
Using the adjoint property, we have in the remapped Lagrangian and Eulerian domain respectively
\begin{align*}
\int_{\MM} \big(\partial_t+ {U}_{t}\big)&\big(\ff\circ\Phi_{t}\big)(\xxb) \psi(t,\xxb{\big|} {s}  ,\yyb)d\xxb \\
&=  { \int_{\MM}\big(\ff\circ\Phi_{t}\big)(\xxb) \big(-\partial_t+ {U}_{t}^{\dagger}\big)\psi ({t}  ,\xxb{\big|} {s}  ,\yyb) d\xxb   },
\end{align*}
and
\begin{align*}
 \int_{\LRdt}  \big(\partial_t+ \WH{U}_{t}\big)\ff(\xxh) \WH{\psi}(t,\xxh\big| s,\yyh)  d\xxh&=\int_{\LRdt}\ff(\xxh)  \big( -\partial_t+\WH{U}_t^{\dagger}\big) \WH{\psi} ({t}  ,\xxh\big| s,\yyh)  d\xxh.
\end{align*}
Using the previous Lemma, the forward Kolmogorov equation of the NELD~\eqref{forward-kolmo} follows.
\end{proof} 
Note that   the probability density of  $\XXh_t$ 
\begin{align*}
 {\nu}(t,\xxh )=\int_{\LRdt}\WH{\psi}(t,\xxh\big|s,\yyh)   {{\nu} }(s,\yyh )d\yyh =\int_{\MM}\psi(t,\xxb{\big|} {s}  ,\yyb) \big({\nu}\circ\Phi_{s}\big)(s,\yyb )d\yyb
\end{align*}
  satisfies the forward Kolmogorov equation, thus we denote   probability density   of  $\XXb_t$  by  
\begin{align*}
 \big({\nu}\circ\Phi_{t}\big)(t,\xxb ) =\int_{\MM}\psi(t,\xxb{\big|} {0}  ,\yyb)\nu_{0}(\yyb)d\yyb=\int_{\LRdt}\WH{\psi}(t,\xxh\big| 0,\yyb) {\nu}_{0}(\yyb)d\yyb,
\end{align*}
where   $\nu_{0}= \nu(0,\xxb ) $.
Moreover, the backward evolution  from  $\big({\phi}\circ\Phi_{t}\big)(t, \xxb)$ to $\big({\phi}\circ\Phi_{s}\big)( s,\yyb)$
satisfies 
\begin{align*}  
 \big({\phi}\circ\Phi_{s}\big)( s,\yyb)=&\int_{\MM} \big({\phi}\circ\Phi_{t}\big)( t,\xxb) {\psi}(t,\xxb {\big|} {s}   ,\yyb)d \xxb .
\end{align*}   
\subsection{Convergence of NELD to a Limit Cycle}\label{sectProposition} 
We show in this section that the Markov process $ \XXb_{t}  $ converges to a limit cycle in the remapped Lagrangian domain, and with  the following lemma that the convergence of $\XXb_t$ to a limit cycle implies the convergence of $\XXh_{t}$ to a limit cycle in the remapped Eulerian domain as well.
\begin{lemma}\label{bijection}
  Let us assume that the Markov process $ \XXb_{t}  $ 
 converges uniformly to a limit cycle $ {\psi}$ in the remapped Lagrangian domain and that $ {\psi}$ is smooth and positive.  Then  the Markov process $ \XXh_{t}  $  converges uniformly to a  probability density function  $\WH{\psi}$ in the remapped Eulerian domain. 
\end{lemma}
\begin{proof} 
Since the density ${\psi} $  is   smooth and positive by assumption, we have
\begin{align*}  
 \EEys{(\ff\circ\Phi_{t})(\XXb_{{t}  })}=&\int_{\MM}(\ff\circ\Phi_{t})(\xxb)  {\psi}(t,\xxb {\big|} {s}   ,\yyb) d \xxb \\
 =&\int_{\LRdt}\ff(\xxh)( {\psi} \circ \Phi_{t}^{-1})(t,\xxh {\big|} {s}   ,\yyh)d \xxh   \\
 =&\int_{\LRdt}\ff(\xxh) \WH{\psi}(t,\xxh {\big|} {s}   ,\yyh) d \xxh ,
\end{align*}
where $\WH{\psi}= {\psi} \circ \Phi_{t}^{-1}$.
In addition, $ \WH{\psi}(t,\xxh {\big|} \cdot)  $ is smooth and positive as ${\psi}(t,\xxb {\big|} \cdot) $ is. 
\end{proof}

Now, let us consider $\xxb_{k}=(Q_{k}=\qqb_{kT},P_{k}=\ppb_{kT})\in\MM$, so that  $(\xxb_{k})_{k\geq 0 }$ denotes a discrete  Markov chain constructed from the particle coordinates at the start of each each period, where   $\xxb_0=(Q_{0},P_{0})$ is the initial coordinate of the time-inhomogeneous process $\xxb_{t}=({\qqb}_{t},{\ppb}_{t})$. Then, we  define 
\begin{align*}
( \mathcal{U}_{T}\ff)(\qqb,{\ppb})=\mathbb{E}\Big(\ff(Q_{k+1},P_{k+1})|(Q_{k},P_{k})=(\qqb,{\ppb})\Big),
\end{align*}
the discrete generator of the Markov chain. 
   In addition, we consider the Lyapunov function 
\begin{align}\label{Lyapunov}
{\KK}_n({\qqb},{\ppb})=1+\norm{\ppb}^{2 n}, n\geq 1 
\end{align}
with the associated weighted $L^{\infty}$ norms  defined by
\begin{align*}
\norm{\mathbf{g}}_{ L_{{\KK}_n}^{\infty}}=\norm{\frac{\mathbf{g}}{{\KK}_n}}_{L^{\infty}}, \quad \mathbf{g}({\qqb}, {\ppb}) \in \MM,
\end{align*}
and the corresponding $\norm{\mathbf{g}}_{L^{\infty}( L_{{\KK}_n}^{\infty})}$ norms defined by
\begin{align*}
\norm{\mathbf{g}}_{L^{\infty}( L_{{\KK}_n}^{\infty})}=\sup_{\theta\in \toro} \normLK{ \mathbf{g}(\theta)}, \quad  \mathbf{g}(\theta, {\qqb}, {\ppb}) \in \toro  \times\MM .
\end{align*}
In the remainder of this section, we prove the following propositions:
\begin{proposition}{(Uniform convergence to a limit cycle)}
\label{prop:converge}
For $n > 1$, there exists a unique
probability measure $ \big(\nu\circ\Phi_{\theta}\big)(\theta,{\qqb}, {\ppb})$ on $ \toro  \times\MM $
and constants $C_n, \lambda_n$  such that, for any initial distribution $({\qq}_0, {\pp}_0)$, we have:
\begin{enumerate}
\item  \label{unique}Exponential convergence
\begin{align*}
\forall \ff \in L^{\infty}( L_{{\KK}_n}^{\infty}), \quad \norml{ \mathbb{E}\Big((\ff\circ\Phi_{t})( {\qqb}_{t},\ppb_{t})\Big)-\WB{\ff}([t])}\leq C_n e^{-\lambda_n t }\norm{\ff-\WB{\ff}([t])}_{ L^{\infty}( L_{{\KK}_n}^{\infty})},
\end{align*}
where, for $\theta \in\toro$, the spatial average of $\ff$ reads
\begin{align}\label{limit-cycle}
\WB{\ff}(\theta)=\int_{ \MM } \big(\ff\circ\Phi_{\theta}\big)( {\qqb},{\ppb})\big(\nu\circ\Phi_{\theta}\big)(\theta,{\qqb},{\ppb}) d{\qqb}d{\ppb}.
\end{align}
\item \label{smooth}The invariant distribution $ \big(\nu\circ\Phi_{\theta}\big)(\theta,{\qqb}, {\ppb})$ is smooth, positive,   and satisfies the Fokker-Planck equation
\begin{align*}
(-\partial_{{\theta}  }+{U}^{\dagger}_{\theta})\big(\nu\circ\Phi_{\theta}\big) =0, \quad \int_{ \toro  \times\MM}\big(\nu\circ\Phi_{\theta}\big)(\theta,{\qqb},{\ppb}) d\theta d{\qqb}d{\ppb} =T .
\end{align*}
\item \label{bound}The invariant distribution $ \big(\nu\circ\Phi_{\theta}\big)(\theta,{\qqb}, {\ppb})$ has finite moments of order $2n$ uniform in the time variable
\begin{align*}
\theta \in \toro, \quad  \int_{\MM }{\KK}_n( {\qqb},{\ppb})\big(\nu\circ\Phi_{\theta}\big)(\theta,{\qqb},{\ppb}) d{\qqb}d{\ppb}\leq R_n <\infty
\end{align*}
and has uniform marginals in the time variable:
\begin{align*}
\WB{\nu}(\theta )=\int_{\MM }  \big(\nu\circ\Phi_{\theta}\big)(\theta,{\qqb},{\ppb}) d{\qqb}d{\ppb}=1.
\end{align*}
\end{enumerate}
\end{proposition} 
The proof of Proposition~\ref{prop:converge} is completed once we show  the {smoothness and positivity of the transition probability} (Section~\ref{smoothness}), the uniformity of the Lyapunov condition and the uniform minorization conditions of the generator of the Markov chain (Section~\ref{existence}).
Then, we derive the  convergence in the Law of Large Numbers of the  Markov chain  in Section~\ref{Convergence}.

Using the above result, we derive the   following Proposition: 
 \begin{proposition}\label{convergence lln}
Let us consider $\ff\in {L^{\infty}( L_{{\KK}_n}^{\infty})}$. From all initial positions,  we have
\begin{align}\label{lln-cont}
\frac{1}{t}\int_0^t (\ff\circ\Phi_{s})({\qqb_s},{\ppb_s})ds\xrightarrow[{t\to +\infty}]{} \int_{\MM}(\ff\circ\Phi_{t})({\qqb},{\ppb})(\nu\circ\Phi_{t})(t,\qqb,\ppb ) d{\qqb}d{\ppb}  \quad \textrm{a,s.}
.
\end{align}
 \end{proposition}
 \subsubsection{Smoothness and positivity of the transition probability}\label{smoothness}
First, we show the smoothness of the transition kernel $\nu_{t}(\cdot) $  in  the remapped Eulerian coordinates using  \cite[Lemma 22.2.5]{hormander}, then we show its positivity using \cite{luc-ergo}. 

 Let $ H_{s}^{loc}$ denote the local Sobolev
space of index $s$.
We will use the following Lemma:
  \begin{lemma}\cite[Lemma 22.2.5]{hormander} \label{horm}
   If ${\WH{U}^{\dagger}_t}$ is hypoelliptic then
\begin{align*}
     {\WH{U}^{\dagger}_t}   {g}&= {h}  \textrm{ and } {h} \in H_{s}^{loc}  \textrm{ at  } (\xxh,\yyh)\in (\LRdt,\LRdt) \implies   {g} \in H_{s+\epsilon}^{loc} \textrm{ at  } (\xxh,\yyh).
\end{align*}
 \end{lemma}
Then it follows that: 
 \begin{corollary}\label{smooth-density}
  Let us assume that $ {\WH{U}^{\dagger}_t}$ is hypoelliptic and   that there exists $\nu_{t} $ such that
\begin{align*}
 \big( -\partial_{t} +\WH{U}^{\dagger}_t\big) \nu_{t}  =0 .
\end{align*}  
  Then $\nu_{t}(\cdot) \in C^{\infty}$ . 
 \end{corollary}
\begin{proof} 
We observe that if ${\WH{U}^{\dagger}_t} $ is hypoelliptic then, 
  Lemma~\ref{horm} shows that $\nu_{t}(\cdot)  \in C^{\infty}$. 
\end{proof}
We finish the first part of the section by showing that $\WH{U}_t$ is hypoelliptic in the following Lemma:
\begin{lemma}
$\WH{U}_t, \WH{U}^{\dagger}_t$ are hypoelliptic.
\end{lemma}
\begin{proof}
We rewrite the generator of the NELD in   the  time-periodic domain as follows:
 \begin{align*}
  {\WH{U}_t}= \WH{\mathcal{X}}_0 +\frac{1}{2}\sum_{i=1}^d \WH{\mathcal{X}}_i,  \textrm{ for } 1\leq i\leq d,   \textrm{ and }
\end{align*}   
\begin{align*}
 \WH{\mathcal{X}}_0 =\left\langle\pph+A{\qqh}, \nabla_{\qqh}\right\rangle+\left\langle-\nabla {V}(\qqh)   ,\nabla_{\pph}\right\rangle-\gamma\left\langle\pph, \nabla_{\pph}\right\rangle, \  \WH{\mathcal{X}}_i=\sqrt{\frac{2\gamma}{\beta}}\partial_{\pph_i}.
\end{align*}
Then, we define $\mathfrak{L}( \WH{\mathcal{X}}_0,\dots, \WH{\mathcal{X}}_d)$, the  Lie algebra of the family of the vectorial space operators $( \WH{\mathcal{X}}_0,\dots, \WH{\mathcal{X}}_d) \in \textrm{Span}( \WH{\mathcal{X}}_0,\dots, \WH{\mathcal{X}}_d)$   satisfying  the stability property:
   \begin{align*}
    {B}\in \mathfrak{L}( \WH{\mathcal{X}}_0,\dots, \WH{\mathcal{X}}_d)\implies [{B}, \WH{\mathcal{X}}_i]\in \mathfrak{L}( \WH{\mathcal{X}}_0,\dots, \WH{\mathcal{X}}_d), \quad i=0,\dots,d,
\end{align*}
     where the Lie bracket between two operators $\mathscr{C}$ and $\mathscr{D}$ is 
     \begin{align*}
     [\mathscr{C}, \mathscr{D}]=\mathscr{C} \mathscr{D}-\mathscr{D}\mathscr{C}.
\end{align*}
     Since for every   point $(\qqh,\pph)\in \LRdt$, we have   
      \begin{align*} 
       [ \WH{\mathcal{X}}_i, \WH{\mathcal{X}}_0]&=\sqrt{\frac{\gamma}{2\beta}}(\partial_{\qqh_i}+\gamma) \partial_{\pph_i}, \quad \forall i\in \{1\dots d\},
  \end{align*}   
     evaluated at $(\qqh_0,\pph_0)$ span $\RRRd$, it follows that $\WH{U}_t$ and $\WH{U}^{\dagger}_t$ are hypoelliptic using \cite[Theorem 1.1]{hormander}. Then it follows that  $\WH{U}_t$ and $\WH{U}^{\dagger}_t$ are hypoelliptic as well.
\end{proof}

Now let us prove that the generator in the remapped Lagrangian coordinates has a positive probability density. We consider the following Lemma:
 \begin{lemma} 
 \label{positive-density}
${U}_t$ has a positive transition kernel.
\end{lemma}

\begin{proof}
For $t>0$ and two points $({\qqb}_0,\ppb_0)$ and $({\qqb}_{t} ,\ppb_{t} )$, let us consider $\varphi(t)$
be any $\mathcal{C}^2$ path in $\MM$ which satisfies $\varphi(0)={\qqb}_0$, $\varphi(t)={\qqb}_{t}$, $\varphi^{'}(0)=\ppb_0$, and  $\varphi^{'}(t)=\ppb_{t} $.    Then we can rewrite  the NELD equation as 
\begin{align*}
\mathcal{V}_{t} =\sqrt{\frac{2\beta}{\gamma}}{e^{[t]A}}   \Big( \ddot{\varphi}_{t} +\nabla {V}({e^{[t]A}}    \varphi_{t} )+\AAI\dot{\varphi}_{t}\Big),
\end{align*}
where  $\mathcal{V}_{t} $ is a smooth control. Thus $(\varphi_{t} ,\dot{\varphi}_{t} )$ is a solution of the control system so that, ${U}_{t}$ drives the system from $({\qqb}_0,\ppb_0)$ to $({\qqb}_{t} ,\ppb_{t} )$. This implies that the support of the transition kernel $\mathcal{A}_{t} ({\qqb},\ppb) =\MM, \forall s>0, \  ({\qqb},\ppb)\in \MM. $
Thus by \cite[Corollary 6.2]{luc-ergo}, it follows that the transition  kernel is positive.
\end{proof}
Now, let us denote $\mathcal{B}_{\delta}(x)$ the open ball of raduis $\delta$ centered at $x$. We summarize the results of this section in the following Corollary:
 \begin{corollary} \label{Assumptionmin}
 The Markov process $\xxb$ with transition kernel $\WB{\mathcal{P}}_t(\xxb,\WB{B})$ satisfies, for some
fixed compact set $C \in \mathcal{B}(\MM)$, the following:
 \begin{itemize}
     \item for some $z^{*}  \in int(C)$ there is, for any $\delta > 0$, a $t_1 = t_1(\delta) \in\toro$ such that
     \begin{align*}
         \WB{\mathcal{P}}_{t_1}(\xxb,\mathcal{B}_{\delta}(z^{*})) >0, \quad \forall \xxb\in C 
\end{align*}
         \item for $t \in \toro$ the transition kernel possesses a density $\psi(\xxb,\yyb)$ precisely
         \begin{align*}
             \WB{\mathcal{P}}_t(\xxb,\WB{B}) =\int_{\WB{B}}\psi(\xxb,\yyb) d\yyb ,\ \forall \xxb\in C, \ \WB{B}\in \mathcal{B}(\MM)\cap\mathcal{B}(C),
         \end{align*}
and $\psi(\xxb,\yyb)$ is jointly continuous in $(\xxb,\yyb)\in C\times C$.
 \end{itemize} 
  
 \end{corollary}
\begin{proof}
 The proof of the first argument is based on  the positivity of the transition kernel from Lemma~\ref{positive-density}  and the second is based on the smoothness of density from Corollary~\ref{smooth-density}.
\end{proof}
We use the above Corollary in the next section to show that the Lyanpunov function satisfies the minorization condition.
\subsubsection{The Invariant Measure of the Discrete Process}\label{existence}
The convergence of the Markov chain $\big(Q_{k+1},P_{k+1}\big)$ is   based on the uniform Lyapunov condition \cite[Assumpion 1]{Hairer} and the uniform minorization condition \cite[Assumpion 2]{Hairer} that we prove in the following two Lemmas.
\begin{lemma}{ (Uniform Lyapunov condition) } \label{Lyap}
There exists $a_n\in [0,1)$ and $b_n>0$ such that 
\begin{align}\label{unique-discetre}
\mathcal{U}_T {\KK}_n \leq a_n {\KK}_n+b_n,
\end{align}
for $\KK_n$ defined in~\eqref{Lyapunov}.
\end{lemma}
\begin{proof}
First, let us recall the NELD in the remapped Lagrangian coordinates under the remappings as follows:
 \begin{align} \label{NELD-22} 
 \begin{cases}
d   \qqb_{kT+\theta} &=  \ppb_{kT+\theta} {d\theta} ,\\
d\ppb_{kT+\theta}&= - e^{ - \theta A }  \nabla {V}\big(e^{ \theta A }   \qqb_{kT+\theta}\big)  {d\theta} -\Gamma\ppb_{kT+\theta} {d\theta}+   \sigma  e^{ - \theta  A }  dW_{kT+\theta},
\end{cases} 
\end{align} 
where $(\theta,\qqb_{kT+\theta},\ppb_{kT+\theta}) \in{ \toro \times\MM}$, and $\AAI=(\gamma+A)$.
Multiplying the second equation of \eqref{NELD-22} by the integrating factor $e^{\AAI\theta}$, we get
we have 
\begin{align}\label{eq2-2}  
d\big({e^{\AAI \theta}{\ppb}_{kT+\theta}}\big)
&=   e^{\gamma\theta}(-  {\nabla {V}(e^{\theta A}  {\qqb}_{kT+\theta}) } {d\theta} +\sigma dW_{kT+\theta} ) 
.\end{align}
We  integrate \eqref{eq2-2} over a period to get the evolution of $\ppb$ 
up to, but not including the remapping, finding
\begin{align*} 
e^{\AAI  T }{P}_{k+1}^{-}-  {P}_k = \int_0^{T}e^{\gamma\theta}( -   {\nabla {V}(e^{\theta A}  {\qqb}_{kT+\theta}) }  {d\theta} +\sigma   dW_{kT+\theta} ).
\end{align*}
We multiply the above equation by $e^{- \AAI  T }$ to get
\begin{align*}
 {P}_{k+1}^{-}&= e^{-  \AAI  T }({P}_k+ \mathscr{F}_k+ \mathscr{G}_k ),
\end{align*}
where
\begin{align*}
\mathscr{F}_k&= \int_0^{T} e^{ \gamma \theta  } {\nabla {V}(e^{\theta A}  {\qqb}_{kT+\theta}) }{d\theta},\quad
\mathscr{G}_k=  \sqrt{\frac{2 \gamma}{\beta}}\int_0^{T}e^{ \gamma \theta  } dW_{kT+\theta}.
\end{align*}
We apply the remapping to the momentum to obtain
\begin{align*}
{P}_{k+1}&=e^{ TA} {P}_{k+1}^{-} =e^{-  \gamma  T } ({P}_k+ \mathscr{F}_k+ \mathscr{G}_k ).
\end{align*}
Letting $\alpha =   e^{-\gamma T  } $, we have $ 0\leq \alpha <1$. Also, note  that $\mathscr{G}_k$ has mean zero and covariance  $ \frac{1 }{\beta}\big({e^{2 \gamma T}-1}\big) $ as
\begin{align*}
\EEF{ \mathscr{G}_k^2 } &=   \frac{1}{\beta}  \EEF{  2 \gamma \big({\int_0^{T}e^{ \gamma \theta  }\big)
dW_{kT+\theta}}^2  } 
=\frac{1 }{\beta}\big({e^{2 \gamma T}-1}\big).
\end{align*}                                       
Since $\normll{ e^{ \gamma\theta }  \nabla {V}\big({e^{ \theta A }   \qqb_{kT+\theta}}\big)} \leq C$, we have
\begin{align*}
\normll{\mathscr{F}_k} \leq \alpha \normll{e^{-AT}} T C   <\infty.
\end{align*}
Using the previous results, we get
\begin{align*}
\normll{ P_{k+1}}^2&\leq   \alpha^2\normll{  P_k+\mathscr{F}_k+\mathscr{G}_k }^2\leq \alpha^2 (\normll{P_k}^2+  \normll{ \mathscr{F}_k+ \mathscr{G}_k }^2+2  \left\langle P_k, \mathscr{F}_k+ \mathscr{G}_k \right\rangle )\\
&\leq  \alpha^2((1+ \mu) \normll{P_k}^2+ \big({2+\frac{1}{ \mu }}\big) \normll{\mathscr{F}_k}^2+2  \left\langle P_k,  \mathscr{G}_k \right\rangle+2\normll{\mathscr{G}_k}^2 ),
\end{align*}
where we choose $ \mu$ s.t. $  \alpha^2(1+ \mu)<1 $. Since  $\EEF{  \left\langle P_k, \mathscr{G}_k \right\rangle} =0$, taking the expectation of $\KK_2$,  
we get
\begin{align*}
\EEF{ {\KK}_2(P_{k+1},Q_{k+1})} \leq    \alpha^2(1+ \mu) {\KK}_2(P_{k },Q_{k })+C_{ \mu},
\end{align*}
where $C_{ \mu}=\big({2+\frac{1}{4 \mu }}\big) \normll{\mathscr{F}_k}^2+2\normll{\mathscr{G}_k}^2$.
This  leads us to find a bound on the expectation of $\KK_n$ as  
\begin{align*}
\EEF{ {\KK}_n(P_{k+1},Q_{k+1})} \leq \big( { \alpha^2(1+ \mu) {\KK}_2(P_{k },Q_{k })+C_{ \mu}}\big)^n \leq  \alpha^{2n}(1+ \mu)^n\normll{P_{k }}^{2n}+\mathfrak{p}(| P_{k }|),
\end{align*}
where $\mathfrak{p}$ is a polynomial of degree at most $2(n-1)$, with positive coefficient. Since we have
\begin{align*}
\frac{\mathfrak{p}(| P_{k }|)}{ \normll{P_{k }}^{2n}}\to 0, \textrm{ as }\normll{P_{k }}\to \infty \implies
\mathfrak{p}(| P_{k }|)=\delta \normll{P_{k }}^{2n}+ C_{\delta},
\end{align*}
and for $\delta$ small, finally we get the uniform bound
\begin{align*}
\EEF{ {\KK}_n(P_{k+1},Q_{k+1}) } \leq  a\normll{P_{k }}^{2n}+C_{\delta}.
\end{align*}
where $a=  \alpha^{2n}(1+ \mu)^n+\delta<1.$ 
\end{proof} 
\begin{lemma}{(Uniform minorization condition)}\label{Uniform}
Fix any $p_{max}>0$, then there exists a probability measure $\vartheta\in \MM$ and a constant $\kappa$ such that, 
\begin{align*}
\forall \WB{B}\in \mathcal{B}( \MM), \quad  \mathbb{P}\Big((P_{k+1},Q_{k+1})\in \WB{B} \ \Big| \ \normll{P_{k }}\leq p_{max}\Big)&\geq \kappa \vartheta(\WB{B}).
\end{align*}  
\end{lemma}
The proof is essentially based on the arguments from \cite{luc-ergo,Mattinglya2002} which uses the continuity property of the Markov process,  the irregularity and positivity of the transition kernel.  
Before we start the proof, we  consider the following Lemma:
\begin{lemma}\cite[Lemma 2.3]{Mattinglya2002}
If the Markov process $\xxb$ satisfies the assumption in Corollary~\ref{Assumptionmin},  then there is a choice of $t\in\toro$, an $\kappa\ge 0$, and a probability measure $\vartheta$, with $\vartheta(C^{c})=0$ and $\vartheta(C) = 1$, such that
\begin{align}
    \WB{\mathcal{P}}(\xxb,\WB{B})\geq \kappa \vartheta(\WB{B}), \quad \forall \WB{B}\in \mathcal{B}(\MM), \ \xxb\in C.
\end{align} 
\end{lemma} 
Then the proof of Lemma~\ref{Uniform} is following:
\begin{proof}
Since the discrete chain $(P_{k+1},Q_{k+1})\in \WB{B} $ satisfies the assumption in Corollary~\ref{Assumptionmin}, then using the above Lemma, we have the expected result.
\end{proof}

Using the  previous Lemmas, we state the following uniform convergence result
for the sampled chain $(Q_k,P_k)$ from \cite{Hairer}:
\begin{theorem}\cite[Theorem 1.2]{Hairer} \label{hairerT}
If   $\mathcal{U}_T$ satisfies the Lyapunov condition as in  Lemma~\ref{Lyap} and the minorization condition as in Lemma~\ref{Uniform}, then  $ \mathcal{U}_T$ admits a unique invariant measure 
$ {\nu }$ such that  for   $C_n,\lambda_n > 0$,
\begin{align*}
\WB{\ff} =\int_{ \MM } \ff({\qqb},{\ppb}) \nu ( {\qqb},{\ppb})d{\qqb} d{\ppb},
\end{align*}
\begin{align}\label{equ-uniform}
\normLK{\mathcal{U}_T^k \ff-\WB{\ff} }\leq  C_n  e^{-\lambda_n k T} \normLK{\ff-\WB{\ff}}.
\end{align}
\end{theorem}   
 
Then, we derive the convergence of the continuous process  $(\qqb_{t},\ppb_{t})$ in the following Lemma: 
 \begin{lemma}\label{uniform-continuous}
The Markov process $(\qqb_{t},\ppb_{t})$ converges exponentially  to the  limit cycle $ {\nu\circ\Phi_{t}}$:
\begin{align*}
\norml{\EEys{\Big({\ff\circ\Phi_{t}}\Big)(\xxb_{t  })}- \WB{\ff }([t])}\leq  C_n  e^{-\lambda_n t}\normLK{\ff-\WB{\ff }([t])}\Big( 1+\KK_n (\yy) \Big), \quad   \yy=\xx_0,
\end{align*}
where $\WB{\ff}([t])$ is defined in \eqref{limit-cycle}.
 \end{lemma}
 \begin{proof}
We work in $(\qqh,\pph)$ variables to simplify the proof.  We use an argument from \cite{meynTweedie,Mattinglya2002} and the result from Theorem~\ref{hairerT} to show that the Markov process $(\qqh_{kT+\theta},\pph_{kT+\theta})\in \LRdt$ converges exponentially  to a limit cycle. The result implies immediately that the Markov process $(\qq_{kT+\theta},\pp_{kT+\theta})\in\MM$ converges as well. We start by using the result from Lemma~\ref{Lyap} and Lemma~\ref{Uniform}, and derive from Theorem~\ref{hairerT} that 
\begin{align*}
\norml{\EEys{\ff(\xxh_{{kT}  })}-\WB{\ff} }\leq  C_n  e^{-\lambda_n k T}\normLK{\ff-\WB{\ff}}\KK_{n  } (\yy ) ,
\quad
\WB{\ff} =\int_{  \WH{\LL}_{\theta}^d\times\RRRd } {\ff  }({\qqh},{\pph})\nu( {\qqh},{\pph})d{\qqh} d{\pph},
\end{align*}
 where we rewrite \eqref{equ-uniform} in remapped Eulerian coordinates. Conditioning on $\mathcal{F}_{\theta}$, we have
\begin{align}\label{EQ-1}
\norml{\EEys{\ff(\xxh_{{kT+\theta}  })}- \WB{\ff}(\theta)  }\leq  C_n  e^{-\lambda_n k T}\normLK{\ff- \WB{\ff}(\theta)}\EEys{\KK_{n  } (\xxh_{\theta} )  }.
\end{align}
We compute an upper bound on $\EEys{\KK_{n  } (\xxh_{\theta} )  }$ as follows:
using the  It\^o's lemma, we get
\begin{align*}
d\ff(\XXh_{{t}  }) = \big(  \WH{U}_t \ff\big)(\XXh_{{t}  })dt+\left\langle\nabla\ff(\XXh_{{t}  }),\WT{\Sigma} {dW} \right\rangle, 
\end{align*}
where
\begin{align*}
 {\WH{U}_t}=\left\langle\pph+A{\qqh}, \nabla_{\qqh} \cdot\right\rangle+\left\langle-\nabla V(\qqh) ,\nabla_{\pph}\cdot\right\rangle-\gamma\left\langle\pph, \nabla_{\pph}\cdot\right\rangle  +\frac{1}{2} \sigma  \sigma^T:\nabla^2 .
\end{align*}
Then, it follows that
 \begin{align*}
 {\WH{U}_t}\KK_n(\qqh,\pph)= & 
    - n\normll{\pph}^{n-2}\left\langle\nabla V( \qqh)  , {\pph}\right\rangle -n \gamma\Big( \left\langle  \pph,  {\pph}\right\rangle-\frac{n+d-2}{\beta} \Big)\normll{\pph}^{n-2} \\ 
         \leq & -n  \gamma\normll{\pph}^{ n  } +n\normll{\nabla V( \qqh) }\normll{\pph}^{n-1}+n\gamma \frac{ n+d-2 }{\beta}  \normll{\pph}^{n-2}. 
\end{align*}
Thus, there exists $\WH{a}_n, \WH{b}_n\geq 0$ such that
\begin{align*}
 {\WH{U}_t}\KK_n\leq -\WH{a}_n\KK_n+\WH{b}_n, \quad \WH{a}_n=n  \gamma,\quad \textrm{as }
 \lim_{\norml{\qqh,\pph}\to\infty}\frac{{\WH{U}_t}\KK_n}{\KK_n}\leq -\WH{a}_n,
\end{align*}
and it follows that
 \begin{align*}
 d\KK_n(\xxh_{kT+\theta})\leq (-\WH{a}_n\KK_n+\WH{b}_n){d\theta}+\textrm{Martingle}.
\end{align*}
  Then, we get an upper bound on $\EEys{\KK_{n  } (\xxh_{\theta} )  }$   by using Gr\"{o}nwall's inequality:
\begin{align*} 
\EEys{\KK_n (\xxh_{\theta})} \leq e^{-\WH{a}_n \theta}\KK_n (\yy)+\frac{\WH{b}_n}{\WH{a}_n} \Big( 1- e^{-\WH{a}_n \theta}\Big)\leq e^{-\WH{a}_n \theta}\KK_n (\yy)+\frac{\WH{b}_n}{\WH{a}_n} .
\end{align*}
Plugging the latter result in \eqref{EQ-1}, we have 
\begin{align*}
\norml{\EEys{\ff(\xxh_{{kT+\theta}  })}-\WB{\ff}(\theta)}\leq  C_n  e^{-\lambda_n k T}\normLK{\ff- \WB{\ff}(\theta)} \Big( e^{-\WH{a}_n \theta}\KK_n (\yy)+\frac{\WH{b}_n}{\WH{a}_n} \Big).
\end{align*}
Defining $\lambda_n$ by $e^{-\lambda_n}=\WH{a}_n^{\frac{1}{T}}$, we obtain the expected result by redefining $C_n\to \Big( 1+\frac{\WH{b}_n}{\WH{a}_n}e^{\lambda_n T}\Big)$: 
\begin{align*}
\norml{\EEys{\ff(\xxh_{{kT+\theta}  })}-\WB{\ff}(\theta)}\leq  C_n  e^{-\lambda_n(kT+\theta)}\normLK{\ff- \WB{\ff}(\theta)} \Big( 1+\KK_n (\yy) \Big).
\end{align*}
 \end{proof}

 \subsubsection{Convergence in Law of Large Numbers for $\big({Q_k,P_k} \big)$ }\label{Convergence}
We use Lemmas from previous section and mainly \cite{meyn} to show that    $\big({Q_k,P_k} \big)$ is positive Harris recurrent chain. Thus, the Law of Large Number holds: 
 \begin{proposition}{(Law of Large Numbers for the sampled chain)}\label{prop1} 
For any $\ff\in L_{{\KK}_n}^{\infty}$, 
\begin{align*}
\frac{1}{N}\sum_{k=1}^N \ff(Q_k,P_k)\xrightarrow[{N\to +\infty}]{} \int_{\MM}\ff({\qqb},{\ppb})\nu({\qqb},{\ppb}) d{\qqb}d{\ppb}  \quad \textrm{a,s.}
,
\end{align*}
for all the initial conditions  $(Q_0,P_0)$.
\end{proposition} 
\begin{proof} 
Corollary~\ref{smooth-density} from the previous Section shows that the chain $\big({Q_k,P_k} \big)$ is  irreducible with respect to the Lebesgues measure. Thus, the result from \cite[Corollary 1]{Tierney} based on \cite[Page 199]{meyn} provides that  the chain $\big({Q_k,P_k} \big)$ is  Harris recurrent. In addition as  $\Big(\nu\circ\Phi_{\theta}\Big)(\theta,\qqb,\ppb )$ is positive using Lemma~\ref{positive-density}, it follows from  \cite[Theorem 17.0.1]{meyn} that the Law of large Numbers   holds.  
\end{proof} 

 \subsubsection{Convergence to the Limit Cycle} \label{convergence}
We consider the following Lemma:
\begin{lemma}\label{convergence lc}
The Markov process $(\qqb_{t},\ppb_{t})$ converges   exponentially  to the the limit cycle $ {\nu\circ\Phi_{t}}$
from any initial configuration $(\qqb_{0},\ppb_{0}).$ 
 \end{lemma}
 \begin{proof} 
The result from Corollary~\ref{smooth-density} and Lemma~\ref{positive-density} show that the chain is irreducible with respect to the Lebesgues measure. In addition,   the process $\big({\qqb_t,\ppb_t} \big)$ is  Harris recurrent using the result from \cite[Corollary 1]{Tierney} based on \cite[Page 199]{meyn}. Thus, the  exponential convergence of the chain $(\qqb_{t},\ppb_{t})$ to an invariant measure $\big( {\nu\circ\Phi_{t}} \big)$ defined in Lemma~\ref{uniform-continuous}, hold for all initial positions. Finally, using  \cite[Theorem 17.0.1]{meyn}, it follows  that the Law of large Numbers defined in \eqref{lln-cont}   holds as well.
 \end{proof}
Now we use the above result to  prove Proposition~\ref{prop:converge}.
 \begin{proof}[Proof of  Proposition~\ref{prop:converge}]
 The result from Lemma~\ref{convergence lc} shows    part~\ref{unique}  of   Proposition~\ref{prop:converge}, exponential convergence to the limit cycle of the NELD. We prove part~\ref{smooth} of  Proposition~\ref{prop:converge} using  Lemma~\ref{positive-density}  and the forward Kolmogorov equation of the NELD as follows: since $\nu$ is a positive probability density function,  Lemma~\ref{forward-kolmo-lemma} gives
 \begin{align*}
(-\partial_{{\theta}  }+{U}^{\dagger}_{\theta})\big(\nu\circ\Phi_{\theta}\big) =0.
\end{align*}
The last part of  Proposition~\ref{prop:converge} is shown by integrating the equation~\eqref{unique-discetre} with respect $ \nu$. In fact, iterating \eqref{unique-discetre} gives
\begin{align*}
\EEys{ {\KK}_n(\xxh_{{kT }  })}  \leq e^{- {a}_n kT}\KK_n (\yy)  +\frac{b_n}{1-a_n} \implies \int_{\MM}{\KK}_n\big(\nu\circ\Phi_{\theta}\big) \leq \frac{b_n}{1-a_n}.
\end{align*}
 \end{proof}

Now, we derive the proof of Proposition~\ref{convergence lln}.
\begin{proof}[Proof of Proposition~\ref{convergence lln}]
 The result from the previous proposition shows that the Markov process $\big({\qqb_t,\ppb_t} \big)$ converges exponentially to   $\big( {\nu\circ\Phi_{t}} \big)(t,\qqb,\ppb )$ for all initial positions.  In addition since    the chain $\big({\qqb_t,\ppb_t} \big)$ is  Harris recurrent ( \cite[Corollary 1]{Tierney} based on \cite[Page 199]{meyn}) and $\big(\nu\circ\Phi_{t}\big)(t,\qqb,\ppb )$ is positive using Lemma~\ref{positive-density}, it follows from  \cite[Theorem 17.0.1]{meyn} that the Law of large Numbers   holds.  
\end{proof}  
  \section{Conclusion}
  In this paper, we study the ergodicity of NELD under shear flow and planar elongational flow using respectively LE and   KR Periodic boundary conditions. This is essentially formulated in Proposition~\ref{prop:converge} where, after showing  existence and uniqueness of the limit cycle using a Lyapunov function and a minorization condition, we established the exponential convergence of the Markov   chain generated by the  NELD equation given all  the initial conditions. It will be interesting to establish the   convergence analysis for the three dimensional diagonalizable incompressible flow cases  using the R-KR~\cite{R-KR} algorithm or the GenKR~\cite{Dobson,Hunt} algorithm. However advanced analysis will be needed since the latter PBCs's geometry is not as simple as the current's cases studied.

\section*{Acknowledgements}
The authors would like to thank Gabriel Stoltz for helpful comments on an early version of the manuscript.

\bibliographystyle{siam}

\end{document}